\documentclass[11pt,letterpaper]{article}
\usepackage[margin=1.0in]{geometry}
\usepackage{latexsym,amsmath,amsfonts,amssymb}
\usepackage[utf8]{inputenc}
\usepackage{tikz}
\usepackage{graphicx}
\usepackage{amsthm}
\usepackage{enumitem}
\usepackage{hyperref}
\usepackage[capitalise]{cleveref}
\usepackage{amsrefs}
\usepackage{thmtools}
\usepackage{thm-restate}
\usepackage{caption}
\usepackage{xfrac}

\newtheorem{introtheorem}{Theorem}
\newtheorem{theorem}{Theorem}[section]
\newtheorem{corollary}[theorem]{Corollary}
\newtheorem{proposition}[theorem]{Proposition}
\newtheorem{lemma}[theorem]{Lemma}

\newtheorem{question}[theorem]{Question}

\theoremstyle{remark}

\theoremstyle{definition}
\newtheorem{definition}[theorem]{Definition}

\newtheorem{example}[theorem]{Example}

\newcommand{\llangle}{\langle \langle}
\newcommand{\rrangle}{\rangle \rangle}

\newcommand{\bd}{\partial}

\DeclareMathOperator{\Lab}{Lab}
\DeclareMathOperator{\asdim}{asdim}
\DeclareMathOperator{\asdimAN}{asdim\textnormal{\textsubscript{AN}}}
\DeclareMathOperator{\dimAN}{dim\textnormal{\textsubscript{AN}}}
\DeclareMathOperator{\diam}{diam}

\title{Assouad-Nagata dimension of finitely generated $C'(\sfrac{1}{6})$ groups}
\author{Levi Sledd}

\begin{document}
\maketitle

\begin{abstract}
This paper is the first in a two-part series.  In this paper, we prove that the Assouad-Nagata dimension of any finitely generated (but not necessarily finitely presented) $C'(\sfrac{1}{6})$ group is at most 2.  In the next paper, we use this result, along with techniques of classical small cancellation theory, to answer two open questions in the study of asymptotic and Assouad-Nagata dimension of finitely generated groups.
\end{abstract}

\section{Introduction}

Asymptotic Assouad-Nagata dimension ($\asdimAN$) is a way of defining the dimension of a metric space at large scales, first defined in 1982 by Assouad and influenced by the work of Nagata \cite{Assouad}.  A related, weaker notion of dimension is that of asymptotic dimension ($\asdim$), introduced by Gromov in his landmark 1993 paper \cite{Gromov}.  Since $\asdim$ and $\asdimAN$ are invariant under quasi-isometry, they have naturally become useful tools in geometric group theory: see \cite{Bell_Dranishnikov} or \cite{Bedlewo} for a good introductory survey on asymptotic dimension in group theory, and \cite{Brodskiy_etal} for many corresponding results for Assouad-Nagata dimension.  For finitely generated groups with the word metric, the subject of this paper, \emph{asymptotic} Assouad-Nagata dimension and Assouad-Nagata dimension (usually abbreviated $\dimAN$) are equivalent.  Thus, when talking about finitely generated groups, we use the shorter ``Assouad-Nagata dimension," which we continue to denote by $\asdimAN$.

In this paper, we prove the following theorem.

\begin{introtheorem}\label{IntroThm2}
Every finitely generated $C'(\sfrac{1}{6})$ group has Assouad-Nagata dimension at most $2$.
\end{introtheorem}

Although interesting in its own right, we believe that the real value of this result is that with it one can readily apply techniques of small cancellation theory to the study of asymptotic and Assouad-Nagata dimension of finitely generated groups.  We demonstrate this in a companion paper to this one, where we use \cref{IntroThm2} in order to prove the following.

\begin{introtheorem}\label{IntroThm1}
For every $k, m, n \in \mathbb N \cup \{\infty\}$ with $4 \leq k \leq m \leq n$, there exist finitely generated, recursively presented groups $H$ and $G$ with $H \leqslant G$, such that
\begin{align*}
\asdim(G) &= k\\
\asdimAN(G) &= m\\
\asdimAN(H) &=n \, .
\end{align*}
\end{introtheorem}

\cref{IntroThm1} simultaneously provides the first example of a group with finite asymptotic dimension and finite but greater Assouad-Nagata dimension, as well as the first example of a finitely generated group with a finitely generated subgroup of greater Assouad-Nagata dimension.  Thus we answer two open questions in asymptotic dimension theory (Question (2) of \cite{Higes2} and Questions 8.6-8.7 of \cite{Brodskiy_etal}, respectively).

The importance of \cref{IntroThm2} is that it applies to \textit{infinitely} presented $C'(\sfrac{1}{6})$ groups.  Indeed, in constructing a group satisfying the conclusion of \cref{IntroThm1}, we use several auxiliary $C'(\sfrac{1}{6})$ groups, whose presentations must be infinite for the construction to work.

In the finitely presented case, we have a satisfying classification: since a finitely presented group has asymptotic dimension 1 if and only if it is virtually free \cite{Fujiwara_Whyte, Gentimis}, \cref{IntroThm2} implies that the Assouad-Nagata dimension of a finitely presented $C'(\sfrac{1}{6})$ group is 1 if the group is virtually free, and 2 otherwise.  However, the finitely presented case of \cref{IntroThm2} was likely already known to experts.  Although apparently not in the literature, a MathOverflow post by Agol \cite{Agol} shows how to obtain that $\asdim(G) \leq 2$ for $G$ a finitely presented $C'(\sfrac{1}{6})$ group, using a theorem of Buyalo and Lebedeva that $\asdim(G) = \dim(\bd G)+1$ when $G$ is hyperbolic \cite{Buyalo_Lebedeva}.  One might then wish to derive \cref{IntroThm2} for infinitely presented groups using the same result for finitely presented groups, but this approach cannot work in general.  This is because in \cite{Osajda}, Osajda constructs a sequence of groups and surjective homomorphisms $G_0 \to G_1 \to G_2 \to \cdots$ such that $\asdim(G_n)=2$ for all $n \in \mathbb N$, but the inductive limit of the sequence has infinite asymptotic dimension.

A technique in many proofs relating hyperbolicity and finiteness of asymptotic dimension (see for example \cite{Roe, Bell_Fujiwara, Bowditch, Osin}) is the ``tight geodesics" property, introduced by Bowditch in  \cite{Bowditch} to study the curve graph of a surface of positive complexity.  In this paper we use a similar technique.  Although infinitely presented $C'(\sfrac{1}{6})$ groups are not hyperbolic, they are ``hyperbolic enough" to be susceptible to a kind of tight geodesics argument.  This stems from the fact that geodesic triangles in $C'(\sfrac{1}{6})$ groups have a limited number of specific forms, a result due to Strebel \cite{Strebel}.  Our proof appears to be the first application of a tight geodesics argument in a non-hyperbolic setting.

The paper is organized as follows.  In \cref{asdimPrelims} we review the definitions of asymptotic dimension and asymptotic Assouad-Nagata dimension, and give a version of the Hurewicz mapping theorem for asymptotic Assouad-Nagata dimension used in the next section.  In \cref{Tight}, we introduce the notion of an $(\varepsilon, k)$-tight geodesic combing for $\varepsilon >0$ and $k \in \mathbb N$, and show that a geodesic metric space admitting a $(\varepsilon, k)$-tight geodesic combing for some $\varepsilon >0$ has asymptotic Assouad-Nagata dimension at most $k$.  In \cref{vKDPrelims} we give some preliminaries on van Kampen diagrams and the classical small cancellation condition $C'(\sfrac{1}{6})$.  We also review the classification of van Kampen diagrams over simple geodesic triangles in $C'(\sfrac{1}{6})$ groups due to Strebel, the essential tool needed in the proof of \cref{IntroThm2}.  In \cref{C'asdimAN} we use Strebel's classification to prove that $C'(\sfrac{1}{6})$ groups admit a $(\sfrac{1}{9}, 2)$-tight geodesic combing, and thus have Assouad-Nagata dimension at most 2.

\section{Preliminaries on asymptotic dimension and asymptotic Assouad-Nagata dimension}\label{asdimPrelims} 

In this paper, $0 \in \mathbb N$.  The set of positive integers is $\mathbb Z^+$.  The set of positive real numbers is denoted $\mathbb R^+$, and the set of  non-negative real numbers is $\mathbb R^+_0$.  The letter $d$ always stands for a metric, on whatever set makes sense in context.

Let $X$ be a metric space.  The open ball of radius $r>0$ about a point $x \in X$ is denoted $B(x,r)$.  If $A, B \subseteq X$, then $d(A, B)$ is defined to be $\inf\{d(a,b) \mid a \in A, b \in B\}$, and we write $d(a,B)$ for $d(\{a\}, B)$.  We define $\diam(A) = \sup\{d(a,a') \mid a, a' \in A\}$.

For $D>0$ and $V \subseteq X$, we say that $V$ is $D$-bounded if $\diam(V) \leq D$.  A family $\mathcal V$ of subsets of $X$ is \emph{uniformly bounded by $D$} or \emph{uniformly $D$-bounded} if $\diam(V) \leq D$ for all $V \in \mathcal V$.     For $r>0$, the \emph{$r$-multiplicity} of $\mathcal V$ is the maximum, over all $x \in X$, of the number of elements of $\mathcal V$ having nonempty intersection with $B(x,r)$, if there is a finite maximum: otherwise, we write that the $r$-multiplicity of $\mathcal V$ is $\infty$. 

\begin{definition}\label{asdimDisjointDef}
Let $X$ be a metric space, $n\in \mathbb N$.  The \emph{asymptotic dimension} of $X$ is at most $n$, written $\asdim(X) \leq n$, if for every $r > 0$, there exists an $D(r)>0$ and a cover $\mathcal V$ of $X$ such that $\mathcal V$ has $r$-multiplicity at most $n+1$ and is uniformly bounded by $D(r)$.  The asymptotic dimension of $X$, denoted $\asdim(X)$, is the least $n \in \mathbb N$ such that $\asdim(X) \leq n$, if such an $n$ exists.  Otherwise, we say that $X$ has infinite asymptotic dimension and write $\asdim(X) = \infty$.
\end{definition}

The function $D: \mathbb R^+ \to \mathbb R^+$ is called an \emph{$n$-dimensional control function} for $X$.  We assume without loss of generality that any $n$-dimensional control function is nondecreasing.  Asymptotic Assouad-Nagata dimension is a version of asymptotic dimension in which the control function is required to be linear.

\begin{definition}\label{asdimANDef} \cite{Brodskiy_etal}
Let $X$ be a metric space, $n \in \mathbb N$.  Then the \emph{asymptotic Assouad-Nagata dimension} of $X$ is at most $n$, written $\asdimAN(X) \leq n$, if there exist $a,b >0$ such that $D(r) = ar+b$ is an $n$-dimensional control function for $X$.  The \emph{asymptotic Assouad-Nagata} dimension of $X$, denoted $\asdimAN(X)$, is defined to be the least $n \in \mathbb N$ such that $\asdimAN(X) \leq n$, or $\infty$ if no such $n$ exists.
\end{definition}

It is easy to verify that both asymptotic dimension and asymptotic Assouad-Nagata dimension are invariant under quasi-isometry.  Therefore for a finitely generated group $G$ we define $\asdim(G)$ ($\asdimAN(G)$) to be the asymptotic (Assouad-Nagata) dimension of $G$ equipped with the word metric with respect to any finite generating set.  The proof of the main result uses the Hurewicz mapping theorem for asymptotic Assouad-Nagata dimension.  In order to state it, we must  state definitions extending the notion of a control function to maps between metric spaces.

\begin{definition}\cite{Brodskiy_etal}
\label{fControlDef}
Let $X, Y$ be metric spaces, $f:X \to Y$, and $n \in \mathbb N$.  Then $D_f : \mathbb R^+ \times \mathbb R^+ \to \mathbb R^+$ is an \emph{$n$-dimensional control function} for $f$ if for all $s, K > 0$ and $A \subseteq X$, if $f(A)$ is $K$-bounded then there exists a cover $\mathcal V$ of $A$ such that $\mathcal V$ has $s$-multiplicity at most $n+1$ and is uniformly bounded by $D_f(s, K)$.  We say that $\asdimAN(f) \leq n$ if there exist constants $a,b,c > 0$ such that $D_f(s,K) = as+bK+c$ is an $n$-dimensional control function for $f$.
\end{definition}

\begin{definition}
Let $X,Y$ be metric spaces.  A function $f:X \to Y$ is \emph{asymptotically Lipschitz} if there exist constants $a,b > 0$ such that $d(f(x),f(x')) \leq a(d(x,x'))+b$ for all $x,x' \in X$.
\end{definition}

The following result is known as the Hurewicz mapping theorem for asymptotic Assouad-Nagata dimension.

\begin{theorem}
Let $f:X \to Y$ be an asymptotically Lipschitz map between metric spaces.  Then $\asdimAN(X) \leq \asdimAN(f) + \asdimAN(Y)$.
\end{theorem}

\section{Tight geodesic combings}\label{Tight}

Let $X$ be a metric space.  A subspace $Y \subseteq X$ is called \emph{cobounded} if there exists a constant $c>0$ such that $d(x, Y) \leq c$ for all $x \in X$.

\begin{definition}\label{geodCombingDef}
Let $X$ be a geodesic metric space with base point $x \in X$.  Then a \emph{geodesic combing} of the pointed metric space $(X,x)$ is a set $T = \{T_y \mid y \in Y\}$, where $Y$ is a cobounded subset of $X$ and $T_y$ is a geodesic from $x$ to $y$ for each $y \in Y$.
\end{definition}

Whenever $\Gamma$ is a connected graph, directed or otherwise, we assume that any edge of $\Gamma$ may be traversed contrary to its orientation, and that $\Gamma$ is equipped with the combinatorial metric, so that $\Gamma$ is naturally a geodesic metric space.

\begin{example}\label{GeodSpanningTree}
Suppose that $\Gamma$ is a connected graph equipped with the combinatorial metric, and let $x \in V(\Gamma)$ be a base point.  A geodesic tree rooted at $x$ is a subgraph $T$ of $\Gamma$ such that $T$ is a tree, and for all $y \in V(\Gamma)$, the unique path from $x$ to $y$ in $T$ is geodesic in $\Gamma$.  If $T$ is a geodesic tree rooted at $x$ and $V(T)=V(\Gamma)$, then we call $T$ a geodesic spanning tree rooted at $x$.  If $T$ is a geodesic spanning tree rooted at $x$ and $y \in V(\Gamma)$, let $[x,y]$ be the path from $x$ to $y$ in $T$.  Then $\{[x,y] \mid y \in V(\Gamma)\}$ is a geodesic combing of $(\Gamma, x)$.
\end{example}

Suppose that $\{T_y \mid y \in Y\}$ is a geodesic combing of a pointed geodesic metric space $(X,x)$.  For each $y \in Y$ and $s>0$, let
$$T(y,s) = \bigcup \{T_{y'} \mid y' \in Y \cap B(y,s)\}$$
and for each $t \geq 0$, let
$$S(t) = \{x' \in X \mid d(x,x') = t\}$$
be the sphere of radius $t$ centered at $x$ in $X$.
 
\begin{definition}
Let $(X,x)$ be a pointed geodesic metric space, $Y$ a cobounded subset of $X$, and $T = \{T_y \mid y \in Y\}$ a geodesic combing of $(X,x)$. Let $\varepsilon > 0$ and $k \in \mathbb N$.  Then we say that $T$ is \emph{$(\varepsilon, k)$-tight} if for all $r>0$, $y \in Y$, and $t \leq d(x,y)-r$, we have $|T(y, \varepsilon r) \cap S(t)| \leq k$.
\end{definition}

Figure \ref{TightGeodesicCombingFigure} illustrates this definition.

\begin{figure}[h!]
\centering
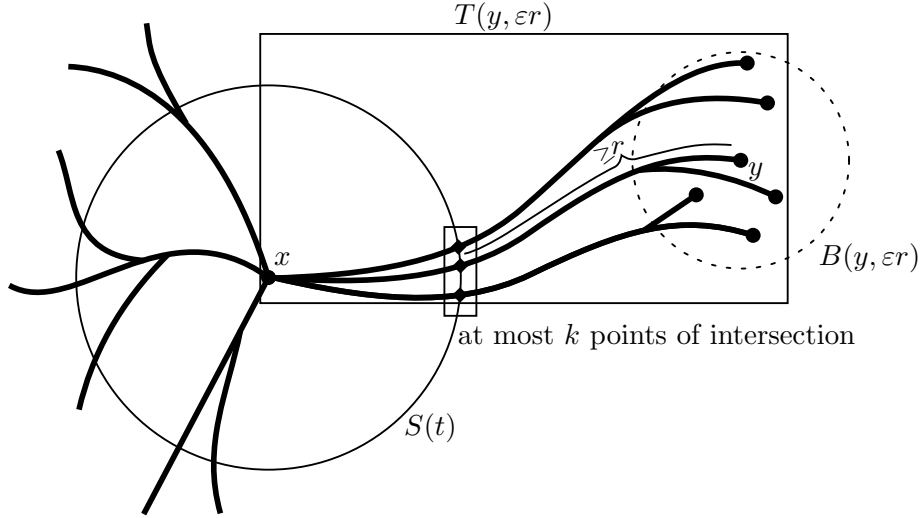
\caption{An $(\varepsilon, k)$-tight geodesic combing.}
\label{TightGeodesicCombingFigure}
\end{figure}

\begin{proposition}\label{TightCombing}
Let $(X,x)$ be a pointed geodesic metric space.  If $X$ admits an $(\varepsilon, k)$-tight geodesic combing for some $\varepsilon > 0$, then $\asdimAN(X) \leq k$.
\end{proposition}

\begin{proof}  Suppose that $Y$ is a cobounded subset of $X$ and $T = \{T_y \mid y \in Y\}$ is a $(\varepsilon, k)$-tight geodesic combing of $(X,x)$.  Let $d_x: Y \to \mathbb R_0^+$ be defined by $d_x(y) = d(x,y)$.  For any $n \in \mathbb N$ and $r > 0$, let
$$A(n,r) = \{y \in Y \mid nr \leq d(x,y) \leq (n+2)r\} = d_x^{-1}([nr, (n+2)r])$$
be the $n^\text{th}$ annulus of width $2r$ in $Y$.  

We claim that for each $n \in \mathbb N$ and $r > 0$, there exists a cover $\mathcal V(n,r)$ of $A(n,r)$ which has $\varepsilon r$-multiplicity at most $k$ and is uniformly bounded by $6r$.  To see this, define an equivalence relation $\sim$ on $A(n,r)$ by declaring that $y \sim y'$ if $T_y$ and $T_{y'}$ pass through the same element of $S((n-1)r)$.  Let $\mathcal V(n,r)$ be the set of $\sim$ equivalence classes.  Clearly $y \sim y'$ implies that there is a path in $T_y \cup T_{y'}$ from $y$ to $y'$ of length at most $6r$, hence $\mathcal V(n,r)$ is uniformly $6r$-bounded.  Furthermore, since $T$ is $(\varepsilon, k)$-tight, for each $y \in A(n,r)$ we have that $|T(y, \varepsilon r) \cap S((n-1)r)| \leq k$, hence any open ball of radius $\varepsilon r$ in $A(n,r)$ can meet at most $k$ equivalence classes.

Now we claim that $\asdimAN(d_x) \leq k-1$.  Let $s, K > 0$ be given.  Now fix $r = \max(\frac{1}{\varepsilon}s, K)$.  Let $A \subseteq Y$ be such that $d_x(A)$ is $K$-bounded.  Then $A \subseteq A(n,K) \subseteq A(n,r)$.  By the previous argument, there exists a cover $\mathcal V(n,r)$ of $A(n,r)$ (and thus of $A$) with $\varepsilon r$-multiplicity at most $k$, which is uniformly bounded by $6r$.  Therefore $\mathcal V(n,r)$ has $s$-multiplicity at most $k$ and is uniformly bounded by $6r = 6\max(\frac{1}{\varepsilon}s, K) \leq \frac{6}{\varepsilon}s+6K$.  Thus $D_{d_x}(s,K) := \frac{6}{\varepsilon}s +6K$ is a $(k-1)$-dimensional control function for $d_x$ that is linear in both $s$ and $K$, and we have $\asdimAN(d_x) \leq k-1$.

It is easy to check that $\asdimAN(\mathbb R_0^+) \leq 1$, and that $d_x$ is 1-Lipschitz and therefore asymptotically Lipschitz.  Therefore by the Hurewicz mapping theorem for asymptotic Assouad-Nagata dimension,
$$\asdimAN(Y) \leq \asdimAN(d_x) + \asdimAN(\mathbb R_0^+) = (k-1)+1 = k.$$

Since $Y$ is quasi-isometric to $X$, $\asdimAN(X) \leq k$.
\end{proof}

A straightforward application of Zorn's Lemma shows that if $\Gamma$ is a connected graph and $x \in V(\Gamma)$, then $\Gamma$ has a geodesic spanning tree rooted at $x$.  Hence \cref{GeodSpanningTree} shows that every connected graph has a geodesic combing, which may or may not be $(\varepsilon, k)$-tight for some $\varepsilon > 0$ and $k \in \mathbb N$.  Since every quasigeodesic metric space is quasi-isometric to a connected graph, \cref{GeodSpanningTree} is more general than it appears at first glance.

Clearly if $\Gamma$ is a connected graph, $x \in V(\Gamma)$, and $\Gamma$ admits a $(\varepsilon, k)$-tight geodesic combing for some $\varepsilon > 0$ and $k \in \mathbb N$, then we may assume without loss of generality that it is given by a geodesic spanning tree.  If $T$ is a geodesic spanning tree of $\Gamma$ rooted at $x$, we say that $T$ is $(\varepsilon, k)$-tight if the geodesic combing it induces is $(\varepsilon, k)$-tight.  In \cref{C'asdimAN} we show that if $\Gamma$ is the Cayley graph of a finitely generated $C'(\sfrac{1}{6})$ group with respect to any finite generating set, then any geodesic spanning tree of $\Gamma$ is $(\sfrac{1}{9}, 2)$-tight.

\section{Preliminaries on van Kampen diagrams and small cancellation}\label{vKDPrelims}

We assume that the reader is familiar with van Kampen diagrams and the $C'(\lambda)$ condition.  However, in the literature there are myriad definitions of van Kampen diagram, each with subtle differences. In addition, our definition of `piece' (and thus, of the $C'(\lambda)$ condition), though clearly equivalent, is not the way it's usually stated.  This is in order to ensure that certain concepts in the companion paper (namely signed and unsigned face counts) are well defined.  Therefore in \cref{classicalSC} and \cref{vanKampenDiagrams} we fix terminology and notation, for use this paper and its sequel.  In summary: we treat a van Kampen diagram as a plane graph, and we include inessential edges and faces in our definition.  It is assumed that presentations are \emph{not} closed under cyclic shifts and inverses, and a piece is defined, not as a common prefix of two words, but as a common prefix of cyclic shifts of two words or their inverses.  If this summary is enough for the reader, they may choose to skip to \cref{simpleGeodTriangles}, referring to Sections \ref{classicalSC} and \ref{vanKampenDiagrams} should the need arise.

In \cref{simpleGeodTriangles}, we present a classification of van Kampen diagrams over simple geodesic triangles in $C'(\sfrac{1}{6})$ groups.  This result, due to Strebel, is the essential tool used in the proof of the main theorem.  Then we prove some lemmas regarding the geometry of simple geodesic triangles in $C'(\sfrac{1}{6})$ groups that are used repeatedly in \cref{C'asdimAN}.

\subsection{The $C'(\lambda)$ condition}\label{classicalSC}

Let $S$ be a set.  Let $S^{-1}$ be the set of formal inverses of $S$, let $1$ be a new symbol not in $S$, and declare $1^{-1}=1$.  Let
\begin{equation}
\begin{split}
S_1 &= S \cup \{1\}\\
S_\circ &= S \cup S^{-1} \cup \{1\}.
\end{split}
\end{equation}

The length of a word $w$ in the free monoid $S_\circ^*$ is denoted $|w|$.  There is a unique word of length 0 called the \emph{empty word} and denoted $\varepsilon$.  We define $w^0$ to be $\varepsilon$ for any $w \in S_\circ^*$.  A word $w \in S_\circ^*$ is \emph{reduced} if $w$ does not contain a subword of the form $1, ss^{-1},$ or $s^{-1}s$ for any $s \in S$, and \emph{cyclically reduced} if every cylcic shift of $w$ (including $w$ itself) is reduced.  

Let $R$ be a language over the alphabet $S_\circ$, that is, $R \subseteq S_\circ^*$.  Then  $R_*$ denotes the closure of $R$ under taking cyclic shifts and formal inverses of its elements.  We say that $R$ is \emph{reduced} if every element of $R$ is reduced, and \emph{cyclically reduced} if $R_*$ is reduced.  We say that $R$ is \emph{cyclically minimal} if it does not contain two distinct words, one of which is a cyclic shift of the other word or its inverse.  That is, $R$ is cyclically minimal if $R \cap \{r\}_* = \{r\}$ for each $r \in R$. 

A (group) \emph{presentation} is a pair $\langle S \mid R \rangle$, where $S$ is a set and $R \subseteq S_\circ^*$.  The notation $G=\langle S \mid R \rangle$ means that $\langle S \mid R \rangle$ is a presentation and $G \cong F(S)/\llangle R \rrangle$, where $F(S)$ is the free group with basis $S$, and $\llangle R \rrangle$ is the normal closure of $R$ as a subset of $F(S)$.  

Whenever $S$ is a generating set of a group $G$, there is a natural monoid epimorphism from $S_\circ^*$ to $G$ that evaluates a word in $S_\circ^*$ as a product of generators and their inverses, and sends $1$ to the identity element.  If $G$ and $S$ are understood, then for a word $w \in S_\circ^*$, we denote by $\bar w$ the image of $w$ under this epimorphism.  If we are considering multiple groups with generators $S$ but different relations, it helps to include the group in the notation.  Thus if $\bar w = g \in G$, then we may write $w =_G g$.  If $u, v \in S_\circ^*$ we may write $u =_G v$ to mean $\bar u =\bar v$ in $G$.  

If again both $G$ and $S$ are understood, the \emph{word norm} on $G$ with respect to $S$ is denoted $\|\cdot\|$ and defined by
$$\|g\| = \min\{|w| \mid w \in S_\circ^*, w=_G g\} \, .$$
We might also denote the word norm on $G$ with respect to $S$ by $\|\cdot\|_G$ or $\|\cdot\|_S$ if the group or generating set is ambiguous.  A word $w \in S_\circ^*$ is called \emph{geodesic} in $G$ if $|w| = \|\bar w\|$.  If $u,w \in S_\circ^*$, $g \in G$, $w$ is geodesic in $G$, and $w =_G u =_G g$, then $w$ is called a \emph{geodesic representative} of $u$ or of $g$.  If $K,C \geq 0$ are fixed constants, then we say that a word $w \in S_\circ^*$ is \emph{$(K,C)$-quasigeodesic} if $|w| \leq K\|\bar w\| +C$.

Given two words $u, v \in S_\circ^*$, we say that $p$ is a \emph{piece} (of $u$ and of $v$) if there exists $u' \in \{u\}_*, v' \in \{v\}_*$ such that $p$ is a common prefix of $u'$ and $v'$.

\begin{definition}
Let $S$ be a set, $R \subseteq S_\circ^*$ a language, and $\lambda$ a real number with $0 < \lambda < 1$.  Then $R$ satisfies $C'(\lambda)$ if, whenever $u,v \in R$ and $u' \in \{u\}_*, v' \in \{v\}_*$ witness that $p$ is a piece of $u$ and $v$, then either $u'=v'$ or $|p| < \lambda \min(|u|, |v|)$.
\end{definition}

In this case we say that $R$ is a $C'(\lambda)$ language.  If $G$ is a group and $G=\langle S \mid R \rangle$ for some $C'(\lambda)$ language $R$, then $\langle S \mid R \rangle$ is called a $C'(\lambda)$ presentation and $G$ is called a $C'(\lambda)$ group.

\subsection{van Kampen diagrams}\label{vanKampenDiagrams}

Let $\Gamma$ be a connected graph.  By a \emph{path} in $\Gamma$ we mean a combinatorial path, i.e. an alternating sequence of vertices and edges, as opposed to a continuous map from a closed interval.  We allow paths to have repeated edges or vertices: in graph-theoretic terms, our `path' is really a walk.  Points in the interiors of edges generally don't matter to us, so we write $x \in \Gamma$ to mean that $x \in V(\Gamma)$.  Likewise, if $\alpha$ is a path in $\Gamma$, then $x \in \alpha$ means that $x$ is a vertex visited by $\alpha$.  

Let $\Gamma$ be any directed graph, and suppose that $\Lab : E(\Gamma) \to S_1$ (see (1) above) is a function which assigns labels from $S_1$ to the edges of $\Gamma$.  Then we extend $\Lab$ to a map from the set of all paths in $\Gamma$ to $S_\circ^*$ in the following natural way.
\begin{itemize}
\item If $e = (x, y)$ is a directed edge labeled $s$, then $\Lab(x, e, y) = s$ and $\Lab(y,e,x) = s^{-1}$.
\item If $\alpha = (x_0, e_1, x_1, \ldots, x_{n-1}, e_n, x_n)$ is a path, then 
$$\Lab(\alpha) = \Lab(x_0, e_1, x_1) \Lab(x_1, e_2, x_2) \cdots \Lab(x_{n-1}, e_n, x_n).$$ 
\end{itemize}

For a path $\alpha$ we define $\ell(\alpha)$, the length of $\alpha$, to be the number of edges traversed by $\alpha$, counting multiplicity.  Equivalently, $\ell(\alpha) = |\Lab(\alpha)|$.

A \emph{plane graph} is a graph which is topologically embedded in $\mathbb R^2$.  A \emph{face} of a plane graph $M$ is the closure of a connected component of $\mathbb R^2 \smallsetminus M$.  Let $F$ be a face of a finite directed plane graph $M$ with edges labeled by elements of $S_1$.  Choosing a base point $x \in \bd F$ and an orientation counterclockwise $(+)$ or clockwise $(-)$, there is a unique circuit which traverses $\bd F$ exactly once, called the \emph{boundary path} and denoted $(\bd F, x, \pm)$.  If all properties of $(\bd F, x, \pm)$ that we care about are preserved after changing its base point and orientation, then we leave these choices out of the notation and write $\bd F$.  We write $\bd M$ instead of $\bd F$ if $F$ is the unbounded face; from now on, `face' will mean `bounded face' unless otherwise stated. The \emph{boundary label} of $F$ is $\Lab(\bd F, x, \pm)$, sometimes denoted by just $\Lab(\bd F)$.

\begin{definition}
A \emph{van Kampen diagram} over a presentation $\langle S \mid R \rangle$ is a finite, connected, directed plane graph $M$ with edges labeled by elements of $S_1$, such that if $F$ is a face of $M$, then either $\Lab(\bd F) \in R_*$ or $\Lab(\bd F) =_{F(S)} 1$.
\end{definition}

A face $F$ is called \emph{essential} if $\Lab(\bd F) \in R_*$ and \emph{inessential} if $\Lab(\bd F)=_{F(S)} 1$.  If $R$ is cyclically reduced then these cases are mutually exclusive.    A face with boundary label $r \in R$ is called an $r$-face.  An edge is \emph{essential} if it is labeled by an element of $S$, and \emph{inessential} if it is labeled by 1.  We call a van Kampen diagram \emph{bare} if it contains no inessential faces, and \emph{padded} otherwise.  In this paper we will only need to consider bare van Kampen diagrams, although padded van Kampen diagrams will be used extensively in the next.  Generally speaking, one needs to consider inessential edges and faces in order to make precise arguments with van Kampen diagrams.  So, in this paper we give the most general definition of a van Kampen diagram, and note the distinction between bare and padded van Kampen diagrams for future reference.

Let $M$ be a van Kampen diagram, and suppose $F$ and $F'$ are distinct faces of $M$.  Then we say that $F$ \emph{cancels} with $F'$ if there exists an edge $e=(x,y)$ in $\bd F \cap \bd F'$ such that $\Lab(\bd F, x, +) = \Lab(\bd F', x, -)$. Then we have the following geometric interpretation of the $C'(\lambda)$ condition, which follows immediately from the definition.

\begin{lemma}\label{SmallCancellationGeometric}
Let $\langle S \mid R \rangle$ be a presentation where $R$ satisfies $C'(\lambda)$, and let $M$ be a van Kampen diagram over $\langle S \mid R \rangle$.  Suppose that $F, F'$ are essential faces of $M$ and $\alpha$ is a common subpath of $\bd F$ and $\bd F'$.  Then either $F$ and $F'$ cancel, or $\ell(\alpha) < \lambda \min (\ell(\bd F), \ell(\bd F'))$.
\end{lemma}

A van Kampen diagram is called \emph{reduced} if no two of its faces cancel.  A van Kampen diagram is \emph{minimal} if, among all van Kampen diagrams with the same boundary label, it minimizes first the number of essential faces, then the number of inessential faces.  If a van Kampen diagram is minimal, then it is bare and reduced \cite{Lyndon_Schupp}.

Whenever $G$ is a group generated by $S$, the Cayley graph of $G$ with respect to $S$ is denoted $\Gamma(G,S)$.

\begin{lemma}[van Kampen Lemma]\cite{Lyndon_Schupp}\label{vKL}
Let $G = \langle S \mid R \rangle$ and $w \in S_\circ^*$.  Then $w=_G 1$ if and only if there exists a van Kampen diagram $M$ over $\langle S \mid R \rangle$ and $x \in \bd M$ such that $\Lab(\bd M, x, +)=w$.  Furthermore, given $g \in G$, there exists a combinatorial map $f: M \to \Gamma(G,S)$ preserving labels and orientations of edges, such that $f(x)=g$.  In particular, $f$ does not increase distances, i.e. is $1$-Lipschitz.
\end{lemma}

\subsection{Van Kampen diagrams for simple geodesic triangles in $C'(\sfrac{1}{6})$ groups}\label{simpleGeodTriangles}

Let $a,b,c$ be distinct elements of $G=\langle S \mid R \rangle$, and let $[a,b],[b,c],[c,a]$ be fixed geodesics between them in $\Gamma (G,S)$.  Then $[a,b] \cup [b,c] \cup [c,a]$ is called a \textit{geodesic triangle} and denoted $\Delta (a,b,c)$.  We say that $\Delta (a,b,c)$ is a \textit{simple} geodesic triangle if the boundary path $\bd \Delta(a,b,c) := [a,b]*[b,c]*[c,a]$ is a simple closed curve in $\Gamma(G,S)$.  If $\sigma$ is a circuit in $\Gamma(G,S)$ beginning at a group element $g \in G$, we say that $M$ is a van Kampen diagram \emph{for} $\sigma$ if, for some $x \in \bd M$, $\Lab(\bd M, x, +) = \Lab(\sigma)$ and the combinatorial map $f: M \to \Gamma(G,S)$ sends $x$ to $g$.

If $\Gamma$ is a directed graph, the \emph{underlying graph} of $\Gamma$ is the undirected graph obtained by removing the orientation of every edge of $\Gamma$.  If $\Gamma$ is a graph and $e = (x,y)$ is an edge of $\Gamma$, then subdividing $e$ means adding a vertex $z$ and edges $(x,z)$ and $(z,y)$ to $\Gamma$, and removing $e$.  A \emph{subdivision} of $\Gamma$ is a graph obtained from $\Gamma$ by a finite sequence of subdivisions of edges.

\begin{theorem}\cite{Strebel} \label{StrebelsTheorem} Suppose that $G=\langle S \mid R \rangle$, $S$ is finite, $R$ satifies $C'(\sfrac{1}{6})$, $\Delta$ is a simple geodesic triangle in $\Gamma (G,S)$, and $M$ is a minimal van Kampen diagram over $\langle S \mid R \rangle$ for $\bd \Delta$. Then the underlying graph of $M$ is a subdivision of a member of one of the four infinite families of plane graphs depicted in Figure \ref{StrebelsTilings}.
\end{theorem}

\begin{figure}[h!]
\centering
%% Accompanies image file 'SCAD1.pdf' (pdf, eps, ps)
\begingroup%
  \makeatletter%
  \providecommand\color[2][]{%
    \errmessage{(Inkscape) Color is used for the text in Inkscape, but the package 'color.sty' is not loaded}%
    \renewcommand\color[2][]{}%
  }%
  \providecommand\transparent[1]{%
    \errmessage{(Inkscape) Transparency is used (non-zero) for the text in Inkscape, but the package 'transparent.sty' is not loaded}%
    \renewcommand\transparent[1]{}%
  }%
  \providecommand\rotatebox[2]{#2}%
  \newcommand*\fsize{\dimexpr\f@size pt\relax}%
  \newcommand*\lineheight[1]{\fontsize{\fsize}{#1\fsize}\selectfont}%
  \ifx\svgwidth\undefined%
    \setlength{\unitlength}{419.21074677bp}%
    \ifx\svgscale\undefined%
      \relax%
    \else%
      \setlength{\unitlength}{\unitlength * \real{\svgscale}}%
    \fi%
  \else%
    \setlength{\unitlength}{\svgwidth}%
  \fi%
  \global\let\svgwidth\undefined%
  \global\let\svgscale\undefined%
  \makeatother%
  \begin{picture}(1,0.19751402)%
    \lineheight{1}%
    \setlength\tabcolsep{0pt}%
    \put(0,0){\includegraphics[width=\unitlength,page=1]{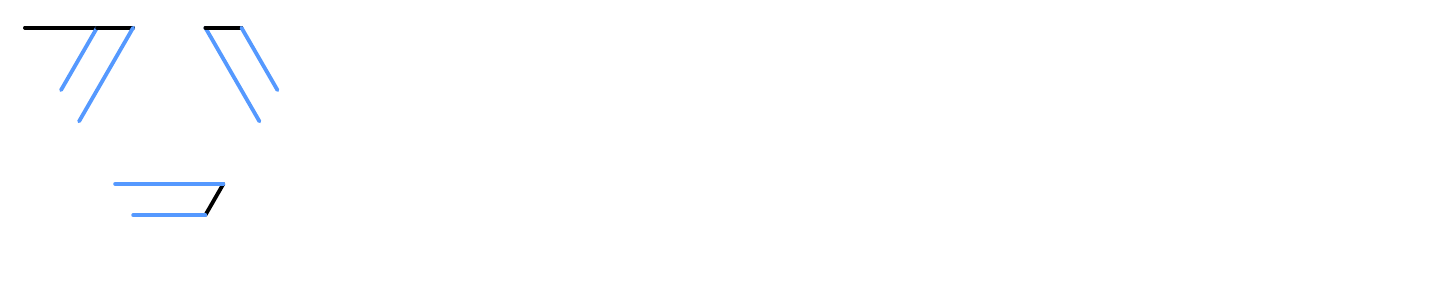}}%
    \put(0.0117492,0.18120532){\color[rgb]{0,0,0}\makebox(0,0)[rt]{\lineheight{1.25}\smash{\begin{tabular}[t]{r}I-II\end{tabular}}}}%
    \put(0,0){\includegraphics[width=\unitlength,page=2]{SCAD1.pdf}}%
    \put(0.27568223,0.18185269){\color[rgb]{0,0,0}\makebox(0,0)[rt]{\lineheight{1.25}\smash{\begin{tabular}[t]{r}III\end{tabular}}}}%
    \put(0,0){\includegraphics[width=\unitlength,page=3]{SCAD1.pdf}}%
    \put(0.53565274,0.18185269){\color[rgb]{0,0,0}\makebox(0,0)[rt]{\lineheight{1.25}\smash{\begin{tabular}[t]{r}IV\end{tabular}}}}%
    \put(0,0){\includegraphics[width=\unitlength,page=4]{SCAD1.pdf}}%
    \put(0.79314413,0.18185269){\color[rgb]{0,0,0}\makebox(0,0)[rt]{\lineheight{1.25}\smash{\begin{tabular}[t]{r}V\end{tabular}}}}%
    \put(0,0){\includegraphics[width=\unitlength,page=5]{SCAD1.pdf}}%
  \end{picture}%
\endgroup%

\caption{types of van Kampen diagrams for a simple geodesic triangle in $\Gamma(G,S)$}
\label{StrebelsTilings}
\end{figure}

In Figure \ref{StrebelsTilings}, the blue edges and dots signify a sequence of parallel edges which may or may not be present.  Vertices are located at the corners and at every juncture of edges.  Our notation is slightly different from Strebel's notation in \cite{Strebel}: our I-II encompasses Strebel's I\textsubscript{2}, I\textsubscript{3} and II, as well as the van Kampen diagram consisting of a single face, and our III is Strebel's III\textsubscript{1}.

For the remainder of this section, suppose that $G$ is a group with presentation $\langle S \mid R \rangle$, $S$ is finite, $R$ satisfies $C'(\sfrac{1}{6})$, $\Delta = \Delta(a,b,c)$ is a simple geodesic triangle in $\Gamma(G,S)$, $M$ is a minimal van Kampen diagram for $\bd \Delta$, $f:M \to \Gamma(G,S)$ is the combinatorial map, $\alpha = [b,c], \beta = [c,a]$, and $\gamma= [a,b]$.  Note that $f|_{\bd M}^\Delta: \bd M \to \Delta$ is bijective, and isometric when restricted to each of the subpaths of $\bd M$ corresponding to $\alpha, \beta,$ or $\gamma$.  Thus without harm we blur the distinction between $\bd M$ and $\bd \Delta$, and refer to vertices, edges, paths etc. in $\bd M$ by their images in $\Delta$.

\begin{figure}[h!]
\centering
%% Accompanies image file 'SCAD9.pdf' (pdf, eps, ps)
\begingroup%
  \makeatletter%
  \providecommand\color[2][]{%
    \errmessage{(Inkscape) Color is used for the text in Inkscape, but the package 'color.sty' is not loaded}%
    \renewcommand\color[2][]{}%
  }%
  \providecommand\transparent[1]{%
    \errmessage{(Inkscape) Transparency is used (non-zero) for the text in Inkscape, but the package 'transparent.sty' is not loaded}%
    \renewcommand\transparent[1]{}%
  }%
  \providecommand\rotatebox[2]{#2}%
  \newcommand*\fsize{\dimexpr\f@size pt\relax}%
  \newcommand*\lineheight[1]{\fontsize{\fsize}{#1\fsize}\selectfont}%
  \ifx\svgwidth\undefined%
    \setlength{\unitlength}{177.1849805bp}%
    \ifx\svgscale\undefined%
      \relax%
    \else%
      \setlength{\unitlength}{\unitlength * \real{\svgscale}}%
    \fi%
  \else%
    \setlength{\unitlength}{\svgwidth}%
  \fi%
  \global\let\svgwidth\undefined%
  \global\let\svgscale\undefined%
  \makeatother%
  \begin{picture}(1,0.62192708)%
    \lineheight{1}%
    \setlength\tabcolsep{0pt}%
    \put(0,0){\includegraphics[width=\unitlength,page=1]{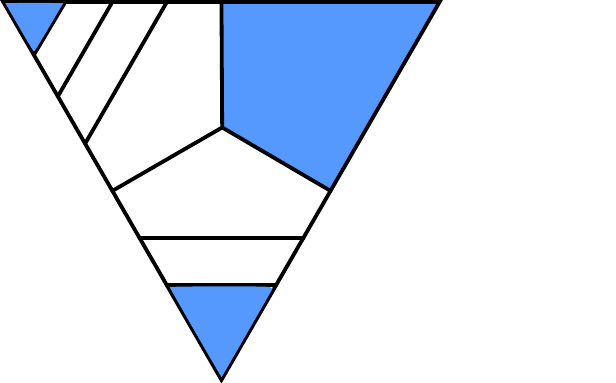}}%
    \put(0.85233751,0.19414751){\color[rgb]{0,0,0}\makebox(0,0)[t]{\lineheight{1.25}\smash{\begin{tabular}[t]{c}extremal face\end{tabular}}}}%
    \put(0,0){\includegraphics[width=\unitlength,page=2]{SCAD9.pdf}}%
    \put(0.35979362,0.27728535){\color[rgb]{0,0,0}\makebox(0,0)[t]{\lineheight{1.25}\smash{\begin{tabular}[t]{c}\small{interior side}\end{tabular}}}}%
    \put(0.87045798,0.39913157){\color[rgb]{0,0,0}\makebox(0,0)[t]{\lineheight{1.25}\smash{\begin{tabular}[t]{c}exterior side\end{tabular}}}}%
    \put(0,0){\includegraphics[width=\unitlength,page=3]{SCAD9.pdf}}%
  \end{picture}%
\endgroup%

\caption{}
\label{typeVExample}
\end{figure}

A face $F$ of $M$ is called \emph{extremal} if $F$ contains $a, b,$ or $c$.  A \emph{side} of $F$ is a maximal subpath of $\bd F$ whose internal vertices all have degree 2 and do not include $a,b,$ or $c$.  A side is called \emph{exterior} if it is contained in $\bd M$, and \emph{interior} otherwise.  An exterior side must be a subpath of $\alpha, \beta$, or $\gamma$, so all exterior sides are geodesic.  We call a face triangular if it has exactly three sides, quadrilateral if it has exactly four sides, etc.  Figure \ref{typeVExample} shows an example of a van Kampen diagram of type V with two triangular faces, four quadrilateral faces, and two pentagonal faces.

Let $i(F)$ denote the number of interior sides of $F$.  The following argument appears so frequently in the proofs that follow that it is worthwhile to section it off as a lemma.

\begin{lemma}\label{DehnReduction}
Let $F$ be a face of $M$ and $\sigma$ an exterior side of $F$.  Then
$$\sum\{ \ell(\tau) \mid \tau \textnormal{ is a side of }F \textnormal{ other than }\sigma\} \geq \tfrac{1}{2}\ell(\bd F).$$
In particular,
$$\sum\{ \ell(\tau) \mid \tau \textnormal{ is an exterior side of }F \textnormal{ other than }\sigma\} > \left(\tfrac{1}{2}-\tfrac{i(F)}{6} \right)\ell(\bd F).$$
\end{lemma}
\begin{proof}
If $\sigma$ is an exterior side, then $\sigma$ is geodesic, from which the first inequality follows.  The second inequality follows from the first inequality and \cref{SmallCancellationGeometric}.
\end{proof}

\begin{definition}
Let $A$ be the union of all faces $F$ of $M$ such that $\bd F$ does not share an edge with $\alpha$, if at least one such $F$ exists: otherwise, set $A = \{a\}$.  We call $A$ the \emph{$a$-corner} of $M$.  Similarly define $B$ and $C$, the $b$-corner and $c$-corner of $M$.  A face which is not included in a corner, i.e. one that shares at least one edge with each of $\alpha, \beta,$ and $\gamma$, is called a \emph{middle face}.  This is unique if it exists, and is denoted $D$.  Thus $A,B,C,D$ divide $M$ into three or four (possibly overlapping) regions.  Figure \ref{Cornersabc} illustrates where the corners and middle faces are in van Kampen diagrams of various types.
\end{definition}

\begin{figure}[h!]
\centering
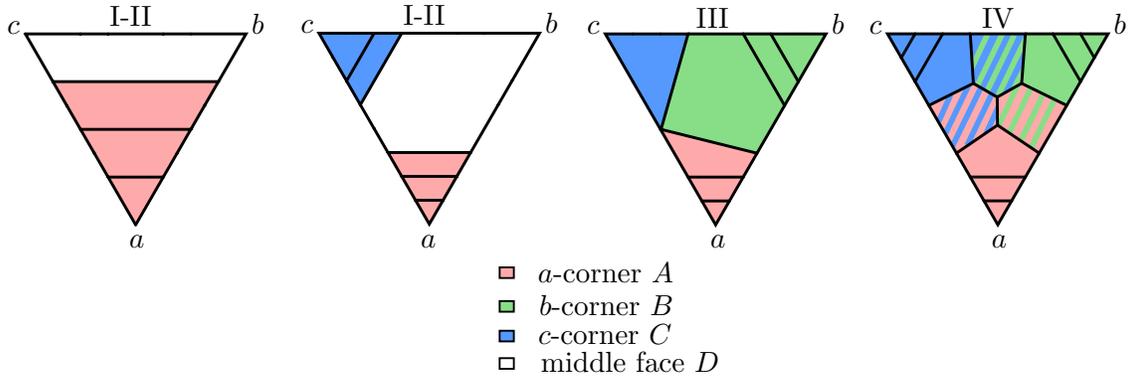
\caption{Examples of corners and middle faces}
\label{Cornersabc}
\end{figure}

A corner may contain no faces if $M$ is of type I-II.  A corner containing at least one face contains an extremal face, which is either triangular, or possibly quadrilateral if $M$ is of type IV or V.  This may be followed by a sequence of quadrilateral faces; which may be followed by a pentagonal face if $M$ is of type III, IV or V; which may be followed by two pentagonal faces, each with one exterior side, if $M$ is of type IV. 

We divide the boundary of the $b$-corner into three parts
\begin{align*}
\alpha_B = \bd B \cap \alpha && \gamma_B = \bd B \cap \gamma && \iota_B = \bd B \smallsetminus (\alpha \cup \gamma)
\end{align*}
and assign similar notation for the other two corners.  The next proposition shows that $\alpha_B, \gamma_B$, and $\iota_B$ are of comparable length, and if one is small, then the entire corner is small.

\begin{proposition}\label{SmallCorners}
The following inequalities hold, and analogous inequalities hold after switching the roles of $a,b,$ and $c$.
\begin{enumerate}[label=\normalfont(\alph*)]
\item $\ell(\iota_B) < 2\min(\ell(\alpha_B), \ell(\gamma_B))$.  If $M$ has a middle face, $\ell(\iota_B) < \min(\ell(\alpha_B), \ell(\gamma_B))$.
\item $\max(\ell(\alpha_B), \ell(\gamma_B)) < 3\min(\ell(\alpha_B), \ell(\gamma_B))$.  If $M$ has a middle face, $\max(\ell(\alpha_B), \ell(\gamma_B)) < 2\min(\ell(\alpha_B), \ell(\gamma_B))$.
\item If $F$ is the middle face of $M$ or $F$ is the pentagonal face of $A$ that borders $B$ and $C$, then $\ell(\bd F \cap (\iota_B \cup \alpha_D \cup \iota_C)) < \ell(\alpha)$.
\end{enumerate}
\end{proposition}
\begin{proof}
Assume that $M$ is of type IV, the most complicated case.  If $M$ is of a different type the arguments are analogous but shorter.  Assume without loss of generality that $\ell(\alpha_B) \leq \ell(\gamma_B)$.

If $B = \{b\}$, then the statement is trivial.  Therefore let $B = \bigcup_{i=0}^{k+3} B_i$, where
\begin{itemize}
\item $B_0$ is the extremal face containing $b$.
\item $B_1, \ldots B_k$ is a (possibly empty) sequence of quadrilateral faces such that $B_{j-1}$ borders $B_j$ for all $j \in \{1, \ldots, k\}$.
\item $B_{k+1}$ is the pentagonal face with two exterior sides, if it exists: otherwise, $B_{k+1} = B_0$.
\item $B_{k+2}$ and $B_{k+3}$ are the pentagonal faces with one exterior side bordering $\alpha$ and $\gamma$, respectively.
\end{itemize}

\begin{figure}[h!]
\centering
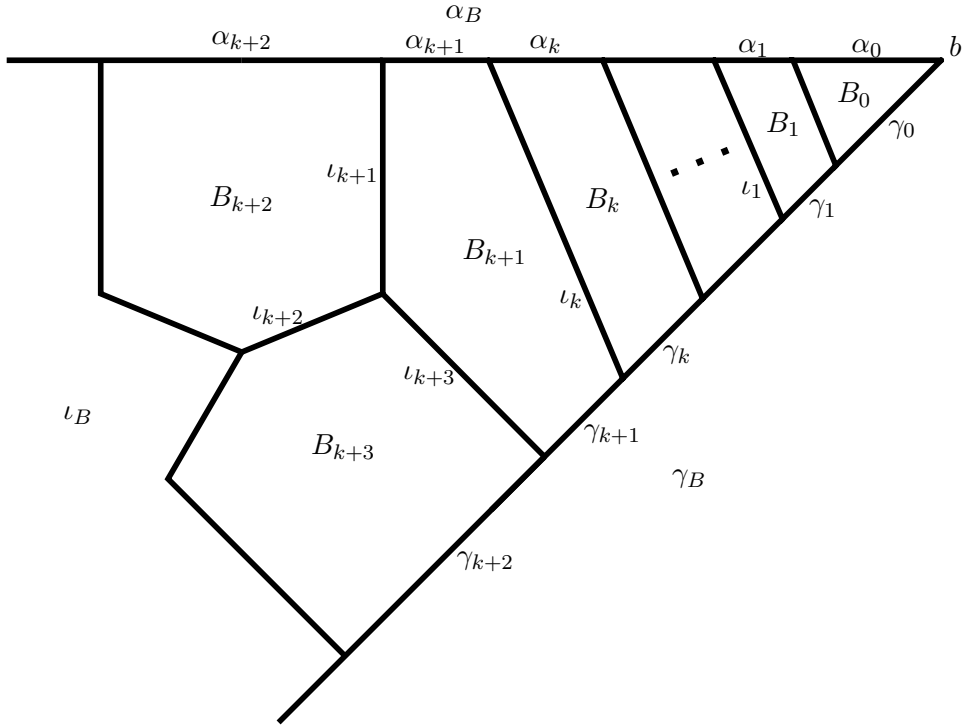
\caption{The $b$-corner of $M$.}
\label{UpperCorner}
\end{figure}

We assign the following labels in order to streamline notation: see Figure \ref{UpperCorner}.

\begin{center}
\begin{tabular}{ll}
$\alpha_i=\bd B_i \cap \alpha$ & for $i \in \{0, \ldots, k+2\}$\\ 
$\gamma_i=\bd B_i \cap \gamma$ & for $i \in \{0, \ldots, k+1\}$, and $\gamma_{k+2} = \bd B_{k+3} \cap \gamma$\\
$\iota_i=\bd B_i \cap \bd B_{i+1}$ & for $i \in \{0, \ldots, k+2\}$, and $\iota_{k+3} = \bd B_{k+1} \cap \bd B_{k+3}$
\end{tabular}
\end{center}

Let $i \in \{0, \ldots, k\}$.  Then applying \cref{DehnReduction} to $\gamma_i$, we obtain $\ell(\alpha_i) > \frac{1}{6}\ell(\bd B_i)$.  Since $\iota_i$ is an interior side of $B_i$, $\ell(\iota_i) < \frac{1}{6}\ell(\bd B_i)$ by the  $C'(\sfrac{1}{6})$ condition.  Therefore
\begin{equation}
\ell(\iota_i) < \ell(\alpha_i) \text{ for } i \in \{0, \ldots, k\} \, .
\end{equation}

Now consider $B_{k+1}$.  Applying \cref{DehnReduction} to $\gamma_{k+1}$ yields $\ell(\alpha_{k+1})+\ell(\iota_k) > \frac{1}{6}\ell(\bd B_{k+1})$.  We know by (2) that $\ell(\iota_k) < \ell(\alpha_k)$, so $\frac{1}{6}\ell(\bd B_{k+1}) < \ell(\alpha_k)+\ell(\alpha_{k+1})$.  But both $\iota_{k+1}$ and $\iota_{k+3}$ are interior sides of $B_{k+1}$.  Therefore $\ell(\iota_{k+1}) < \frac{1}{6}\ell(\bd B_{k+1})$ and $\ell(\iota_{k+3}) < \frac{1}{6}\ell(\bd B_{k+1})$. Thus
\begin{equation}
\begin{split}
\ell(\iota_{k+1}) &< \ell(\alpha_k)+\ell(\alpha_{k+1})\\
\ell(\iota_{k+3}) &< \ell(\alpha_k)+\ell(\alpha_{k+1}) \, .
\end{split}
\end{equation}

Notice that $\bd B_{k+2}$ consists of four interior sides and $\alpha_{k+2}$.  Therefore $\ell(\alpha_{k+2}) > \frac{1}{3}\ell(\bd B_{k+2})$.  Since $\iota_{k+1}$ is an interior side of $B_{k+2}$, we have $\ell(\iota_{k+1})< \frac{1}{6}\ell(\bd B_{k+2})$.  Similar observations about $\iota_{k+1}, \iota_{k+3},$ and $\gamma_{k+2}$ yield
\begin{equation}
\begin{split}
&\ell(\iota_{k+1}) < \tfrac{1}{2}\ell(\alpha_{k+2}) \quad \ell(\iota_{k+3}) < \tfrac{1}{2}\ell(\gamma_{k+2})\\
&\ell(\iota_{k+2}) < \tfrac{1}{2}\ell(\alpha_{k+2}) \quad \ell(\iota_{k+2}) < \tfrac{1}{2}\ell(\gamma_{k+2}) \, .
\end{split}
\end{equation}

Now $\iota_B$ consists of two interior sides of $B_{k+2}$ and two interior sides of $B_{k+3}$. Thus
\begin{equation}
\iota_B < \tfrac{1}{3}\ell(\bd B_{k+2})+\tfrac{1}{3}\ell(\bd B_{k+3}) \, .
\end{equation}

On the other hand, by \cref{DehnReduction} applied to $\alpha_{k+2}$ we have that $\ell(\iota_{k+1})+\ell(\iota_{k+2}) > \frac{1}{6}\ell(\bd B_{k+2})$.  Applying \cref{DehnReduction} to $\gamma_{k+2}$, we find that $\ell(\iota_{k+2})+\ell(\iota_{k+3}) > \frac{1}{6}\ell(\bd B_{k+2})$.  Therefore
\begin{equation}
\tfrac{1}{3}\ell(\bd B_{k+2})+\tfrac{1}{3}\ell(\bd B_{k+3}) < 2\ell(\iota_{k+1})+4\ell(\iota_{k+2})+2\ell(\iota_{k+3}).
\end{equation}

Combining inequalities (3)-(6), we have
$$\ell(\iota_B) < 2\ell(\alpha_k)+2\ell(\alpha_{k+1})+2\ell(\alpha_{k+2}) \leq 2\ell(\alpha_B).$$
If $M$ has a middle face, then $\iota_B = \iota_k$ and inequality (2) gives that $\iota_B < \alpha_B$.  This proves part (a).  Since $\gamma_B$ is geodesic,
$$\ell(\gamma_B) \leq \ell(\iota_B)+\ell(\alpha_B) < 2\ell(\alpha_B)+\ell(\alpha_B) = 3\ell(\alpha_B),$$
and if $M$ has a middle face this bound is lowered to $2\ell(\alpha_B)$.  This establishes part (b).

For part (c), let $\alpha_D = \alpha \cap D$.  If $F = D$, then $\bd F \cap \iota_B = \iota_k$ in Figure \ref{UpperCorner}.  If $F$ is instead the pentagonal face of $A$, then $\bd F \cap \iota_B = \iota_{k+3}$ (similarly for $\iota_C$).  In either case we have $\ell(\bd F \cap (\iota_B \cup \alpha \cup \iota_C)) < \ell(\alpha_B)+\ell(\alpha_D)+\ell(\alpha_C) = \ell(\alpha)$.  
\end{proof}

\section{Proof of the main result}\label{C'asdimAN}

In this section we prove the following proposition.

\begin{proposition}\label{mainProp}
Let $G = \langle S \mid R \rangle$, where $S$ is finite, and $R$ is a cyclically reduced $C'(\sfrac{1}{6})$ language.  Then any geodesic spanning tree of $\Gamma(G,S)$ is $(\sfrac{1}{9},2)$-tight.
\end{proposition}

We divide this section into two parts.  In \cref{vKDConstruction}, we fix all notation and assumptions, and give a description of a van Kampen diagram which is obtained by fixing a geodesic spanning tree of $\Gamma(G,S)$ rooted at 1 and assuming it is \emph{not} $(\varepsilon, 2)$-tight for some $\varepsilon > 0$.  All lemmas in \cref{geodSpanningTreeisTight} are proved under the assumptions stated in \cref{vKDConstruction}.  We determine $\varepsilon$ along the way, choosing at each stage an $\varepsilon$ small enough to make the lemmas work.  In the end we reach a contradiction with any $\varepsilon \leq \frac{1}{9}$, meaning that the spanning tree must have been $(\sfrac{1}{9},2)$-tight all along.

\subsection{Construction of a van Kampen diagram}\label{vKDConstruction}

Let $G=\langle S \mid R \rangle$ be a finitely generated $C'(\sfrac{1}{6})$ group.  Fix a geodesic spanning tree $T$ of $\Gamma(G,S)$ rooted at $1$.  Let let $d$ be the word metric on $G$ with respect to $S$, and let $\|\cdot\|$ be the corresponding word norm.  For each $g \in G$, let $[1,g]$ be the unique path from $1$ to $g$ in $T$.

Suppose to the contrary that $T$ is not $(\varepsilon, 2)$-tight.  Let $r \in \mathbb N$ witness that $T$ is not $(\varepsilon, 2)$-tight.  Then there exists an $x \in G$ such that $\|x\| \geq r$ and $B(x, \varepsilon r)$ contains two elements $y, y'$ such that the geodesics $[1,x], [1,y]$, and $[1,y']$ each pass through different elements of the sphere of radius $\|x\|-r$ in $\Gamma(G,S)$.  Because $T$ is a tree, for every distinct $g, h \in G$, there is a unique vertex of $\Gamma(G,S)$ where the geodesics $[1,g]$ and $[1,h]$ diverge.  Let $a$ be the point at which $[1,x]$ diverges from $[1,y]$, and let $a'$ be the point at which $[1,x]$ diverges from $[1,y']$.  Then we have that $d(a,x), d(a',x) \geq r$ and $d(a,y), d(a',y') > r-\varepsilon r$.  Without loss of generality suppose that $\|a'\| \geq \|a\|$.

Let $[x,y]$ and $[x,y']$ be arbitrarily chosen geodesics.  Let $\Delta (1,x,y)$ be the geodesic triangle in $\Gamma(G,S)$ with sides $[1,x], [1,y]$ and $[x,y]$; similarly define $\Delta (1,x,y')$.  Note that $\Delta(1,x,y)$ is not a tripod, since $\ell([x,y]) < \varepsilon r < 2(r - \varepsilon r) \leq \ell([x,a]) + \ell([a,y])$.  Therefore $\Delta(1,x,y)$ contains exactly one maximal simple geodesic triangle, and $a$ is the vertex of this triangle which is closest to $1$.  Let $\Delta = \Delta(a,b,c)$ be the maximal simple geodesic triangle in $\Delta(1,x,y)$, where $a, b,$ and $c$ are the points closest to $1,x,$ and $y$, respectively.  Similarly let $\Delta' = \Delta (a',b',c')$ be the maximal simple geodesic triangle of $\Delta (1,x,y')$ where $a',b',$ and $c'$ are the vertices of $\Delta (a',b',c')$ which are closest to $1,x,$ and $y'$, respectively.  Note that $d(b,x), d(b',x) < \varepsilon r$, and $d(a', x) > r$, thus we have that $\|a\| \leq \|a'\| < \|b\| \leq \|b'\|$ or $\|a\| \leq \|a'\| < \|b'\| \leq \|b\|$.

\begin{figure}[h!]
\centering
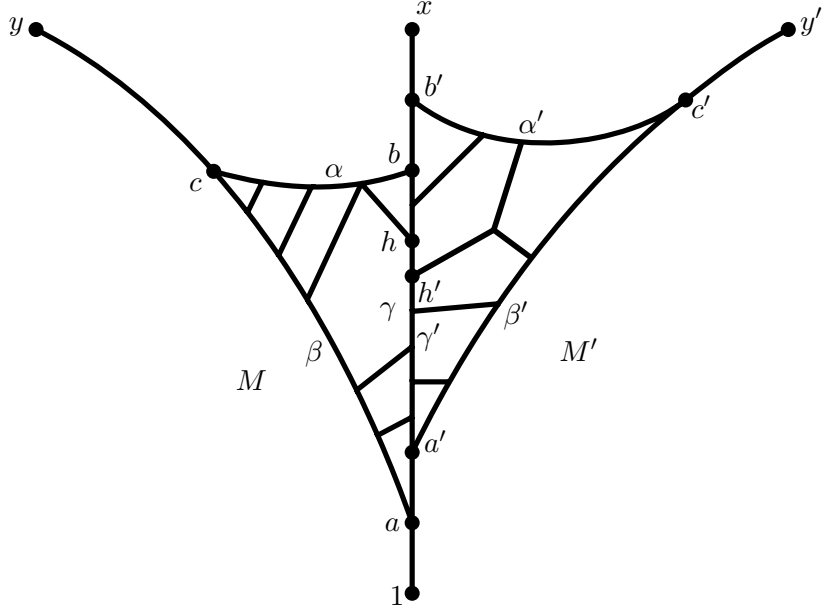
\caption{$N = M \cup_{[1,x]} M'$}\label{GeodTriangles}
\end{figure}

Let $M$ and $M'$ be minimal van Kampen diagrams for $\Delta$ and $\Delta'$, respectively.  Attaching the appropriate geodesic segments and gluing $M$ and $M'$ along $[1,x]$, we obtain a van Kampen diagram, call it $N$, for the circuit $[1,y]*[y,x]*[x,y']*[y',1]$.  Thus $N$ is the diagram shown in Figure \ref{GeodTriangles}, allowing that $b, b'$ may appear in either order along $[1,x]$, and that $M$ and $M'$ may take any of the forms depicted in Figure \ref{StrebelsTilings}.  We retain all notation used in the previous section to describe the geometry of the simple geodesic triangles $\Delta$ and $\Delta'$ and their van Kampen diagrams $M$ and $M'$, using the symbol $'$ where appropriate.  Thus $\alpha$ is the geodesic opposite $a$, $C'$ is the $c'$-corner, which is opposite $\gamma'$, etc.  The vertices labeled $h$ and $h'$ in Figure \ref{GeodTriangles} are the vertices of $\gamma, \gamma'$ at the extremities of the $b$ and $b'$ corner, respectively.

Since $M$ is minimal, no two faces of $M$ cancel: similarly for $M'$.  However, in principle a face of $M$ may cancel with a face of $M'$, so $N$ may or may not be reduced.  If $F$ is a face of $M$, then we refer to the number of sides of $F$ with respect to $M$, not $N$.  Thus for example if $F$ is a quadrilateral face in $M$, we will still refer to it as a quadrilateral face, even though a side of $\bd F$ might be split into multiple sides in $N$ if it borders $[1,x]$.  Likewise, we call a side of a face $F$ of $M$ exterior if it is exterior in $M$, even though it may not be a subpath of $\bd N$.

Note that the combinatorial map $f$ might not be injective when restricted to $\bd M \cup \bd M'$.  For example, it may happen that $f(\alpha')$ intersects $f(\alpha)$ or $f(\beta)$ in $\Gamma(G,S)$.  However, it is important to note that $[1,x]$, $[1,y]$ and $[1,y']$ do not intersect at any vertex of $\Gamma(G,S)$ farther from the identity than $a'$, so $f$ \textit{is} injective when restricted to $\beta \cup \beta' \cup \gamma \cup \gamma'$.

\subsection{Proof that a certain geodesic spanning tree is tight}\label{geodSpanningTreeisTight}

All lemmas in this subsection are proved under the standing assumptions described in \cref{vKDConstruction}, which are not restated.  The argument is as follows.  First, we examine how faces of $M$ and $M'$ may line up along their common boundary, and determine that there is a face of $M'$ that shares more than a third of its boundary with $\gamma$ and does not cancel with any face of $M$.  Playing around with inequalities provided by the $C'(\sfrac{1}{6})$ condition, we find that this situation implies that $\varepsilon > \frac{1}{9}$.  Since we were free to choose $\varepsilon$ from the start, this is the desired contradiction, proving that $T$ is in fact $(\sfrac{1}{9}, 2)$-tight.

\begin{lemma}\label{a'notinCb}
Let $h,h'$ be the vertices of $\gamma_B, \gamma_B'$, respectively, which are closest to $a'$.  Then $\min(d(a',h), d(a',h')) > (1 - 3\varepsilon) r$.
\end{lemma}
\begin{proof}
By \cref{SmallCorners}, $\ell(\gamma_B) < 3\ell(\alpha_B)$ and so $d(h,x) = \ell(\gamma_B)+d(a,x) < 3\ell(\alpha_B)+d(a,x) \leq 3(\ell(\alpha)+d(a,x)) = 3d(b,x) < 3d(y,x) < 3\varepsilon r$.  Since $d(a,x) > r$, we have $d(a,h) > (1-3\varepsilon)r$.  Similarly for $h'$.
\end{proof}

Now we examine how faces of $M$ and $M'$ may meet up along $\gamma \cup \gamma'$.  We say that a face $F$ of $N$ \emph{cancels} if there is some face $F'$ of $M$ such that $F$ and $F'$ cancel.  If $F, F'$ are faces of $M$ and $M'$, we say that $F'$ \emph{subsumes} $F$ if $F$ does not cancel with $F'$ but $(F \cap \gamma) \subseteq (F' \cap \gamma')$.  We say that $F$ is \emph{subsumed} if there is some face $F'$ that subsumes $F$.  We use the same terminology when the roles of $F$, $M$, and $\gamma$ are switched with those of $F'$, $M'$, and $\gamma'$.

\begin{figure}[h!]
\centering
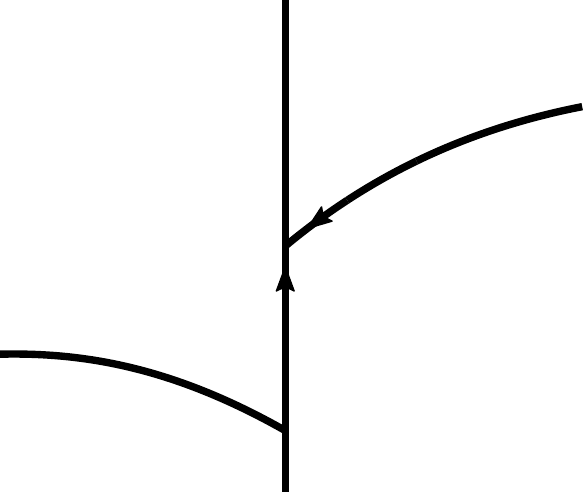
\caption{}
\label{FlushFigure}
\end{figure}

\begin{lemma}\label{Flush}
Let $F, F'$ be faces of $M,M'$.  If $F$ cancels with $F'$, then $\bd F \cap \gamma = \bd F' \cap \gamma'$.
\end{lemma}

\begin{proof}
Suppose that $F$ cancels with $F'$, but $\bd F \cap \gamma \neq \bd F' \cap \gamma'$.  Let $\bd F \cap \gamma = [p,q]$ and $\bd F' \cap \gamma' = [p',q']$.  Then either $p \neq p'$ or $q \neq q'$.  Suppose that $\|p\| < \|p'\|$: the other cases are similar.  Let $\sigma$ be the side of $F$ \emph{in $N$} which is incident to $p'$ and is not contained in $\bd F'$.  Let $\tau'$ be the side of $F'$ incident to $p'$ which is not contained in $\gamma'$: see Figure \ref{FlushFigure}.  Then $\sigma * \tau'$ is a subpath of either a face bordering $F'$ or the geodesic $[1,y']$.  Since $F$ and $F'$ cancel, if $\Lab(\sigma)$ ends with a letter $s$, then $\Lab(\tau')$ begins with $s^{-1}$.  Therefore either the boundary label of some face is not freely reduced, or $\Lab([1,y'])$ is not freely reduced.  The former contradicts the fact that $R$ is cyclically reduced, and the latter contradicts that $[1,y']$ is geodesic.
\end{proof}

\begin{lemma}\label{notInternal}
Let $F$ be a face of $M$ such that either $F$ is triangular, $F$ is quadrilateral, or $F$ is pentagonal with only one exterior side.  Then $F$ is not subsumed.
\end{lemma}

\begin{proof}
Let $\sigma = \bd F \cap \gamma$, and let $\tau$ be the other exterior side of $\bd F$ if there is one, either $\tau = \bd F \cap \alpha$ or $\tau = \bd F \cap \beta$.  If $F$ is triangular or quadrilateral, then applying \cref{DehnReduction} to $\tau$ yields that $\ell(\sigma) > \frac{1}{6} \ell(\bd F)$.  If $F$ is pentagonal and has only one exterior side, then $\sigma$ is the only exterior side of $F$, so $\ell (\sigma) > \ell (\bd F) - \frac{4}{6}\ell(\bd F) = \frac{1}{3}\ell(\bd F)$.  In all cases, if $\sigma$ is also a subpath of the boundary of a face $F'$ which does not cancel with $F$, then this contradicts the $C'(\sfrac{1}{6})$ condition.
\end{proof}

\begin{corollary}\label{Bsubsumed}
If a face $F$ of $N$ is subsumed, then either $F$ is the middle face, or $F$ is a pentagonal face with two exterior sides.  In either case $F$ borders a face in $B$, so $h \in \bd F$.
\end{corollary}

In Lemmas \ref{w'Small}-\ref{DescribingB}, $E'$ is the extremal face at $a'$, and
\begin{align*}
\rho' &= \bd E' \cap \gamma'\\
\sigma' &= \bd E' \cap \beta'\\
\tau' &= \bd E' \smallsetminus (\rho' \cup \sigma').
\end{align*}

\begin{lemma}
\label{w'Small}
If $\varepsilon \leq \frac{1}{6}$, then $\ell(\tau') < \frac{1}{6}\ell(\bd E')$.
\end{lemma}

\begin{figure}[h!]
\centering
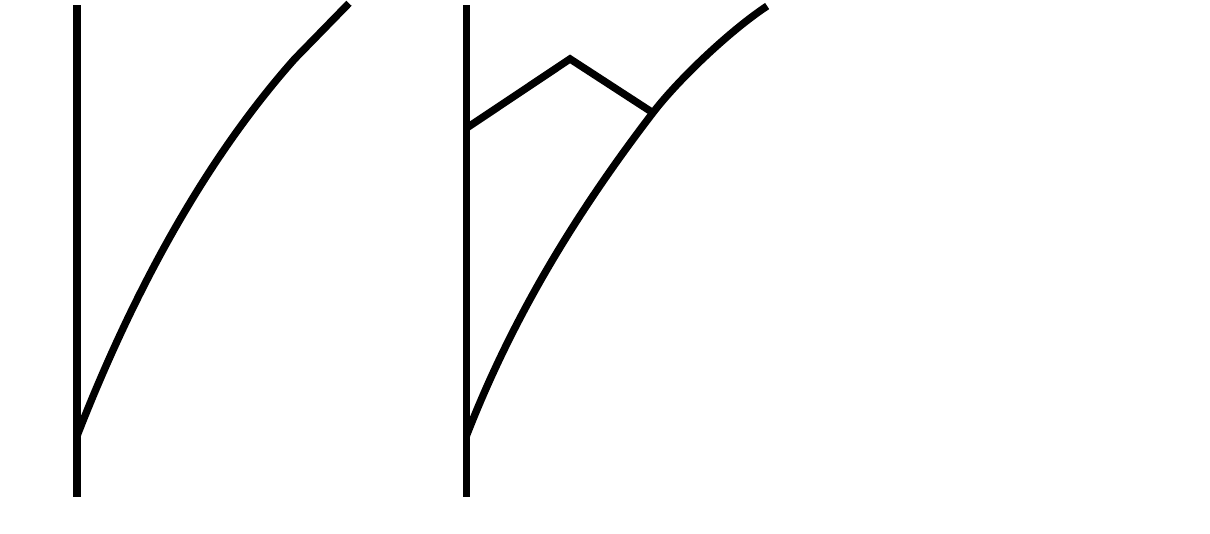
\caption{}
\label{w'SmallFigure}
\end{figure}

\begin{proof}
There are three cases to consider: either $E'$ is triangular (Case 1 in Figure \ref{w'SmallFigure}), $E'$ is quadrilateral (Case 2), or $A'$ contains no faces and $E'$ is the middle face (Case 3).  In Case 1, $\tau'$ is an interior side and the result is immediate.  In Cases 2 and 3, we may apply \cref{SmallCorners} to get that $\ell(\tau')<\ell(\alpha') < \varepsilon r$.  Also, in these cases $E'$ borders $B'$, so $h \in E'$.  Since $a' \in E'$ by definition, $[a',h'] \subseteq \ell(\rho')$ and so $\ell(\bd E')> 2\ell(\rho') \geq 2d(a',h')>2(1-3\varepsilon)r$.  Therefore $\ell(\tau')<\frac{\varepsilon r}{2(1-3\varepsilon)r}\ell(\bd E') = \frac{\varepsilon}{2-6\varepsilon}\ell(\bd E')$.  Solving $\frac{\varepsilon}{2-6\varepsilon} \leq \frac{1}{6}$ yields $\varepsilon \leq \frac{1}{6}$.
\end{proof}

From now on (that is, \cref{MorethanaThirdD'}-\cref{randomCor}), it is assumed that $\varepsilon \leq \frac{1}{6}$, and this will not be restated in the hypotheses.  Knowing \cref{w'Small}, we may then apply \cref{DehnReduction} to $\rho'$ and $\sigma'$ in turn to obtain the following.
\begin{corollary}\label{MorethanaThirdD'}
$\ell(\rho') > \frac{1}{3}\ell(\bd E')$ and $ \ell(\sigma') > \frac{1}{3}\ell(\bd E')$.
\end{corollary}

\begin{lemma}
$E'$ does not cancel.
\end{lemma}

\begin{proof}
Suppose that $E'$ cancels with some face $F$ of $M$.  Since $F \cap \gamma = E' \cap \gamma'$ by \cref{Flush} and $a' \not \in B$, we have that $F$ is not a face of $B$.  Therefore $F$ borders $\beta$.  Now there are two cases.  Either there is an interior side of $F$ which is incident to both $\gamma$ and $\beta$ (Case 1 in Figure \ref{D'notCancelFigure}), or $F$ is the extremal face at $a$ and $a = a'$ (Case 2).  In Case 1, let $\theta$ be the side of $F$ incident to $a'$ and $\beta$.  In Case 2, let $\theta$ be the edge of $\bd F \cap \beta$ which is incident to $a'$.

\begin{figure}[h!]
\centering
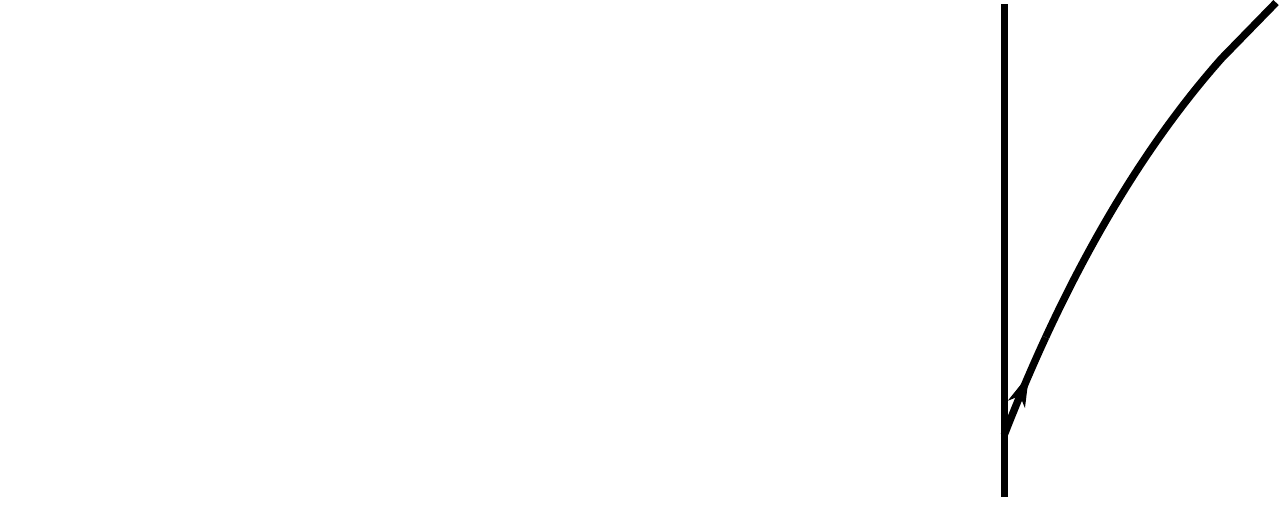
\caption{}
\label{D'notCancelFigure}
\end{figure}

In either case, let $\theta'$ be the path starting from $a'$ which is a subpath of $\bd E'$ and has label $\Lab(\theta)$.  By \cref{MorethanaThirdD'}, $\theta'$ is a subpath of $\sigma'$. Let $p, p'$ be the endpoints of $\theta,\theta'$, respectively.  Since $\Lab(\theta) = \Lab(\theta')$ and the combinatorial map $f$ is label-preserving, $f(p) = f(p')$. Note that $p \in \beta$ by definition, and $p' \in \sigma' \subseteq \beta'$.  Since $T$ is a tree, $f$ is injective when restricted to $\beta \cup \beta'$, so this is a contradiction.
\end{proof}

\begin{lemma}
$\ell (\rho') > (1-3\varepsilon) r$.
\end{lemma}
\begin{proof}
Since $\|a'\| \geq \|a\|$, we have that either $\rho'$ is a subpath of $\gamma$ or $\rho'$ extends beyond $\gamma$.  If the latter is the case, then $a', b \in \rho'$, so $[a, b'] \subset \rho'$ and $\ell(\rho') \geq d(a',b) > (1-\varepsilon) r > (1 - 3\varepsilon) r$.

Suppose then that $\rho'$ is a subpath of $\gamma$. Since $E'$ does not cancel, the $C'(\sfrac{1}{6})$ condition implies that each face of $M$ bordering $E'$ must cover less than one sixth of $\bd E'$.  Recall that $\ell(\rho') > \frac{1}{3}\ell(\bd E')$ by \cref{MorethanaThirdD'}.  Therefore $E'$ must border at least three faces of $M$, so $E'$ subsumes some face $F$.  Since $F$ is subsumed, $(\bd F \cap \gamma) \subseteq \rho'$ and $h \in (\bd F \cap \gamma)$ by \cref{Bsubsumed}.  Since $\rho'$ contains $a'$ as well, we have that $\ell(\rho') \geq d(a',h) > (1-3\varepsilon)r$.
\end{proof}

\begin{corollary}\label{randomCor}
$\ell(\rho' \cap [a',h]) > \frac{1-6\varepsilon}{3-9\varepsilon}\ell(\bd E')$. 
\end{corollary}
\begin{proof}
First, observe that $\ell([h,x]) < 3\varepsilon r$ and $\ell(\rho') > (1-3\varepsilon)r$.  Therefore $\ell([h,x]) < \frac{3\varepsilon}{1-3\varepsilon}\ell(\rho')$, so
$$\ell(\rho \cap [a',h]) = \ell(\rho' \smallsetminus [h,x])\geq \ell(\rho')-\ell([h,x])> \ell(\rho') - \left(\frac{3\varepsilon}{1-3 \varepsilon} \right) \ell(\rho') = \left( \frac{1-6\varepsilon}{1-3\varepsilon} \right)\ell(\rho').$$
By \cref{MorethanaThirdD'}, $\ell(\rho') > \frac{1}{3}\ell(\bd E')$.  Therefore 
$$\ell(\rho' \cap [a',h]) > \left(\frac{1-6\varepsilon}{1-3\varepsilon}\right)\ell(\rho')>\left(\frac{1-6\varepsilon}{3-9\varepsilon} \right) \ell(\bd E').$$
\end{proof}

\begin{lemma}\label{DescribingB}
If $\varepsilon \leq \frac{1}{9}$, there is a face $F$ of $M$ satisfying all of the following conditions.
\begin{enumerate}[label=\normalfont(\alph*)]
\item $F$ is subsumed by $E'$.
\item $F$ is either the middle face of $M$ or the pentagonal face of $A$.
\item $\ell(\bd F \cap \bd E') > \left(\frac{1-6\varepsilon}{3-9\varepsilon}-\frac{1}{6}\right)\ell(\bd E')$.
\item $\ell(\bd F) < 6\varepsilon r$.
\end{enumerate}
\end{lemma}
\begin{proof}
From the previous corollary we know that more than $\frac{1-6\varepsilon}{3-9\varepsilon}$ of $\bd E'$ must be covered by faces which are not in $B$.  Since $E'$ does not cancel, if $\frac{1-6\varepsilon}{3-9\varepsilon} \geq \frac{1}{6}$, then $E'$ subsumes some face $F$ of $M$ which is not in $B$. Solving $\frac{1-6\varepsilon}{3-9\varepsilon} \geq \frac{1}{6}$ yields $\varepsilon \leq \frac{1}{9}$ so choose $\varepsilon \leq \frac{1}{9}$ and part (a) follows.  \cref{notInternal} shows that $F$ can only be the middle face or the pentagonal face of $A$, giving part (b).  Furthermore, no other faces of $M \smallsetminus B$ can be subsumed by $E'$. Now $E'$ subsumes $F$ implies that $\ell(\bd F \cap \bd E') < \frac{1}{6}\ell(\bd E')$, so there is still a subpath of $\rho'$ of length more than $\left(\frac{1-6\varepsilon}{3-9\varepsilon}-\frac{1}{6}\right)\ell(\bd E')$, and thus of positive length, to be covered.  Therefore $E'$ must border one additional face of $M$ which is not contained in $B$.  Call this face $E$.  Since $E$ cannot be subsumed by $E'$, we have that part of $\bd E$ extends beyond $\gamma'$, so $a' \in \bd E$.  Therefore we have the situation depicted in Figure \ref{FinishingArgument}.

\begin{figure}[h!]
\centering
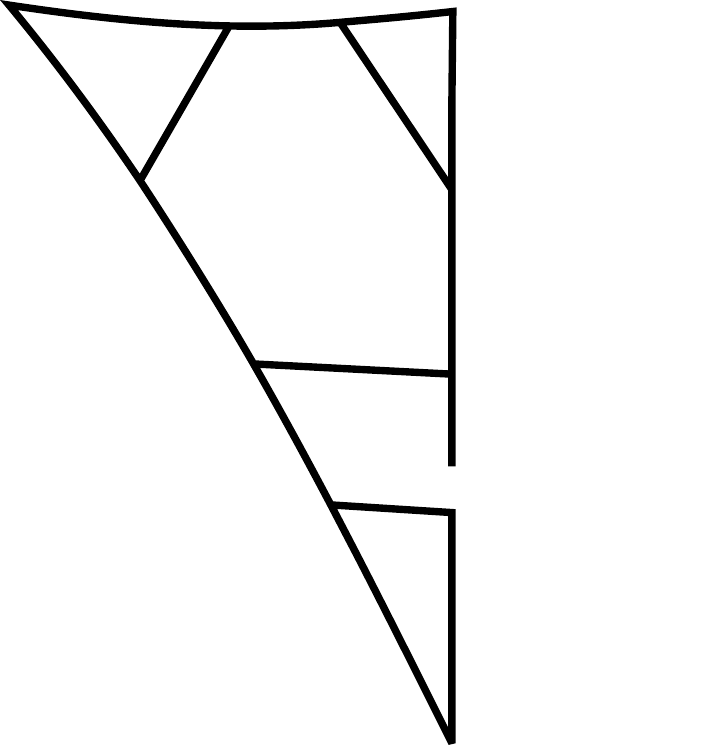
\caption{}
\label{FinishingArgument}
\end{figure}

Notice that $\ell(\bd E \cap \bd E') < \frac{1}{6}\ell(\bd E')$ and $\ell((\bd F \cup \bd E)\cap \bd E') = \ell(\rho' \cap [a',h]) > \frac{1-6\varepsilon}{3-9\varepsilon}\ell(\bd E')$.  Therefore $\ell(\bd F \cap \bd E') > \left(\frac{1-6\varepsilon}{3-9\varepsilon}-\frac{1}{6}\right)\ell(\bd E')$.  This proves part (c).

Let $\rho = \bd F \cap \gamma$ and  $\sigma = \bd F \cap \beta$.  Then $\ell(\rho) < \frac{1}{6}\ell(\bd F)$ since $E'$ subsumes $F$.  Since $F$ is either the middle face or the pentagonal face of $A$, $F$ has exactly one interior side, call it $\tau$, which does not border either $\alpha, B,$ or $C$.  The sum of the lengths of the other sides of $F$ is less than $\ell(\alpha) < \varepsilon r$ by \cref{SmallCorners}. By \cref{DehnReduction} applied to $\sigma$, we have that $\ell(\rho)+\ell(\tau) + \varepsilon r \geq \frac{1}{2}\ell(\bd F)$.  But $\max(\ell(\rho), \ell(\tau)) < \frac{1}{6} \ell(\bd F)$, so we have $\varepsilon r > \frac{1}{6}\ell(\bd F)$, or $\ell(\bd F) < 6\varepsilon r$.  This proves part (d).
\end{proof}

We return to \cref{mainProp}, which we are now ready to prove.
\begin{proof}[Proof of \cref{mainProp}]
Suppose that $T$ is a geodesic spanning tree of $\Gamma(G,S)$ rooted at 1, and $T$ is not $(\varepsilon, 2)$-tight.  If $\varepsilon \leq \frac{1}{9}$, then all of the previous lemmas hold.  But then, in the notation of \cref{DescribingB}, we have 
$$6\varepsilon r > \ell(\bd F) > 6\ell(\bd F \cap \bd E') > 6\left(\tfrac{1}{6}-\tfrac{1-6\varepsilon}{3-9\varepsilon}\right)\ell(\bd E') > \tfrac{9\varepsilon-1}{1-3\varepsilon}2\ell(\rho') = \tfrac{18\varepsilon-2}{1-3\varepsilon}(1-3\varepsilon)r = (18\varepsilon-2)r.$$
Thus $6\varepsilon > 18\varepsilon-2$ or $\varepsilon > \frac{1}{6}$, a contradiction.  Therefore $T$ is $(\sfrac{1}{9}, 2)$-tight.
\end{proof}

Combining this with \cref{TightCombing}, we have the following theorem.

\begin{theorem}\label{mainThm}
If $G$ is a finitely generated $C'(\sfrac{1}{6})$ group, then $\asdimAN(G) \leq 2$.
\end{theorem}

For infinitely generated groups $G$, it is possible to define $\asdim(G)$ to be the supremum of the asymptotic dimensions of its finitely generated subgroups, although this definition doesn't make sense for Assouad-Nagata dimension: see \cite{Bedlewo}.  Of course languages over infinite alphabets can also satisfy $C'(\sfrac{1}{6})$, so the notion of a $C'(\sfrac{1}{6})$ group extends to infinitely generated groups. Therefore we can also say the following.

\begin{corollary}
If $G$ is a $C'(\sfrac{1}{6})$ group, then $\asdim(G) \leq 2$.
\end{corollary}

Combining \cref{mainThm} with results of Fujiwara and Whyte \cite{Fujiwara_Whyte} and Gentimis \cite{Gentimis}, we know that if $G$ is a finitely presented $C'(\sfrac{1}{6})$ group, then $\asdim(G) = \asdimAN(G) = 1$ if $G$ is virtually free, and $\asdim(G) = \asdimAN(G) = 2$ otherwise.  However, for infinitely presented $C'(\sfrac{1}{6})$ groups, the question remains.

\begin{question}
Suppose $G$ is a finitely generated, infinitely presented $C'(\sfrac{1}{6})$ group.  When is $\asdim(G)$ or $\asdimAN(G)$ equal to 1, and when is it equal to 2?
\end{question}

\pagebreak

\end{document}